\documentclass[reqno,11pt]{amsart}
\usepackage[onehalfspacing]{setspace}
\usepackage{mathtools}
\usepackage{amsfonts}
\usepackage{amssymb}
\usepackage{dsfont}
\usepackage{amsthm}
\usepackage{enumerate}
\usepackage{xfrac}
\usepackage{esint}
\usepackage{color}
\usepackage{graphicx}
\usepackage{wrapfig}
\usepackage{caption}

\usepackage[left=2cm,right=2cm,top=2cm,bottom=2cm]{geometry}

\newtheorem{theorem}{Theorem}[section]
\newtheorem{lemma}[theorem]{Lemma}
\newtheorem{proposition}[theorem]{Proposition}
\newtheorem{corollary}[theorem]{Corollary}

\newcommand{\R}{\mathbb R}
\newcommand{\N}{\mathbb N}

\begin{document}
\pagenumbering{arabic}
\title[Fractional regularity for the Brouwer Degree]{Fractional Sobolev regularity for the Brouwer degree}
\author[De Lellis]{Camillo De Lellis}
\address{Institut f\"ur Mathematik, Universit\"at Z\"urich, CH-8057 Z\"urich}
\email{camillo.delellis@math.uzh.ch}

\author[Inauen]{Dominik Inauen}
\address{Institut f\"ur Mathematik, Universit\"at Z\"urich, CH-8057 Z\"urich}
\email{dominik.inauen@math.uzh.ch}

\begin{abstract} We prove that if $\Omega\subset \mathbb R^n$ is a bounded open set and $n\alpha> {\rm dim}_b (\partial \Omega)  = d$, then the Brouwer degree deg$(v,\Omega,\cdot)$ of any H\"older function $v\in C^{0,\alpha}\left (\Omega, \R^{n}\right)$ belongs to the Sobolev space $W^{\beta, p} (\mathbb R^n)$ for every $0\leq \beta <  \frac{n}{p} - \frac{d}{\alpha}$. This extends a summability result of Olbermann and in fact we get, as a byproduct, a more elementary proof of it. Moreover we show the optimality of the range of exponents in the following sense: for every $\beta\geq 0$ and $p\geq 1$ with $\beta > \frac{n}{p} - \frac{n-1}{\alpha}$ there is a vector field $v\in C^{0, \alpha} (B_1, \mathbb R^n)$ with
$\mbox{deg}\, (v, \Omega, \cdot)\notin W^{\beta, p}$, where $B_1 \subset \mathbb R^n$ is the unit ball. 
\end{abstract}
\maketitle

\section{Introduction}

We are interested in the regularity and summability properties of the Brouwer degree of a H\"older continuous function $v\in C^{0,\alpha}\left (\Omega, \R^{n}\right)$ defined on an open, bounded set $\Omega\subset \R^{n}$ with 
\begin{equation}\label{e:box_counting}
n-1\leq d:=\text{dim}_b (\partial \Omega)<n\, ,
\end{equation}
where $\text{dim}_b$ denotes the box counting dimension. In the recent note \cite{Olbermann} Olbermann showed that the Brouwer degree is an $L^{p}$ function for every $1\leq p < \frac{n\alpha}{d}$. A different proof of the $L^1$ summability when $\partial \Omega$ has a Lipschitz boundary has been given independently by Z\"ust in \cite{Zuest}: in fact, although Z\"ust's proof\footnote{There is a gap in the argument of the main result of \cite{Zuest}: however the proof of the $L^1$ estimate for the Brouwer degree, which in that note is regarded as a technical tool, is correct.}  does not yield the range of summability exponents of Olbermann's proof, it allows to conclude the $L^1$ estimate when each component $v^i$ has different H\"older regularity $C^{0, \alpha_i}$ and $\frac{1}{n-1} \sum_i \alpha_i > 1$. We do not know how to modify Olbermann's argument in order to yield the latter conclusion and thus the results in \cite{Olbermann} and \cite{Zuest} complement each other. A natural conjectural generalization of both is that the degree is in $L^p$ under the assumption that $1\leq p < \frac{1}{d} \sum_i \alpha_i$ (a trivial consequence of Olbermann's theorem is
$L^p$ summability for $p < \frac{n}{d}\min_i \alpha_i$). We do not know how to prove such statement but we can at least prove $L^1$ summability under the assumptions that $d= \text{dim}_b (\partial \Omega)\geq n-1$ and $\sum_i \alpha_i > d$ (cf. Theorem \ref{l:Olbermann}). 

\medskip

The most important point of this note is that Olbermann's idea can be improved to show higher (fractional) Sobolev regularity. In particular the following is our main theorem. As usual $[\cdot]_{C^{0,\alpha}}$ denotes the H\"older and $[\cdot]_{W^{\beta,p}}$ the Gagliardo seminorm when $\beta>0$ and the $L^p$ norm for $\beta=0$.

\begin{theorem}\label{t:main}
 Let $\Omega\subset \R^{n}$ be open and bounded, $d$ be as in \eqref{e:box_counting} and $v\in C^{0,\alpha}\left( \Omega,\R^{n} \right)$, where $\alpha\in ]\frac{d}{n},1]$. Then the Brouwer degree $\mbox{\rm{deg}} (v,\Omega, \cdot)$ satisfies the estimate
 \begin{equation}\label{e:degree-estimate-1}
  [ \mbox{{\rm deg}}\, (v,\Omega,\cdot)]_{W^{\beta,p}} \leq C (\Omega,n,\alpha,\beta,p)[ v]_{C^{0,\alpha}}^{\frac{n}{p}-\beta} \,
 \end{equation}
\begin{equation}\label{e:range}
\mbox{for any pair $(\beta,p)$ with}\qquad p\geq 1\qquad \mbox{and}\qquad 0\leq \beta <\frac{n}{p}-\frac{d}{\alpha}\, .
\end{equation} 
\end{theorem}
 Observe that the endpoints of \eqref{e:range} form the segment $\sigma = \{\beta = \frac{n}{p} - \frac{d}{\alpha}\}$ and if we let $(\beta_1, p_1) = (\frac{n\alpha - n+1}{\alpha}, 1)$ be the right extremum of the segment, then $W^{\beta_1, p_1}$ embeds in $W^{\beta, p}$ for every $(\beta, p)\in \sigma$. In particular our theorem has the following obvious corollary.

\begin{corollary}\label{c:convergence}
 Let $\Omega\subset \R^{n}$ be open and bounded, $d$ be as in \eqref{e:box_counting} and $\{v_k\}\subset C^{0,\alpha}\left( \Omega,\R^{n} \right)$ a bounded sequence converging uniformly to $v$, where $\alpha\in ]\frac{d}{n},1]$. Then, for every pair $(\beta, p)$ as in \eqref{e:range}, the sequence $\mbox{{\rm deg}}\, (v_k, \Omega, \cdot)$ converges to ${{\rm deg}}\, (v, \Omega, \cdot)$ strongly in $W^{\beta, p}$.
\end{corollary}

As already mentioned above, our proof is built upon the ideas of Olbermann in \cite{Olbermann}. However we report also a self-contained and more elementary argument for his result: the key simplification can be found in the direct elementary proof of Theorem \ref{l:Olbermann} below. A part of this theorem is shown in \cite{Olbermann} using tools from interpolation theory. We instead derive it directly and use our approach to extend Z\"ust's result in the sense mentioned above. For the reader's convenience we then show how to recover Olbermann's higher integrability in few lines, although the argument is already contained in \cite{Olbermann}. From Theorem \ref{l:Olbermann} we then derive Theorem \ref{t:main} using heavier machinery from harmonic analysis.

\medskip

It has already been shown in \cite{Olbermann} that, when $\beta =0$ and $d>n-1$, the range of exponents in Theorem \ref{t:main} cannot be extended beyond the endpoints: more precisely, \cite[Theorem 1.2]{Olbermann} proves that, if $p> \frac{n\alpha}{d}$, then there is a fixed open set $\Omega$ with $\text{dim}_b (\partial \Omega) = d$ and a bounded sequence $\{v_k\}\subset C^{0, \alpha} (\Omega)$ for which
$\|\text{deg} (v_k, \Omega, \cdot)\|_{L^p} \uparrow \infty$ (note however, that the proof in \cite{Olbermann} {\em does not} yield a $v\in C^{0, \alpha} (\Omega)$ for which $\text{deg} (v, \Omega, \cdot)\not \in L^p$, because the sequence produced by the argument converges to $0$, cf. \cite[Section 4.2]{Olbermann}). In this note we discuss the optimality of the range in the case $d=n-1$: our main conclusion is the following theorem, which, by Sobolev embedding, has the immediate Corollary \ref{c:sharpsobolev}. 

\begin{theorem}\label{t:sharpmoredim}
For any $n\geq 2$, $p\geq 1$ and $\alpha<\frac{p(n-1)}{n}$ there is $v\in C^{0,\alpha}\left( B_1, \R^{n} \right)$ such that $\mbox{\rm deg}\, (v,B_1,\cdot)\notin L^{p}$, where $B_1 \subset \mathbb R^n$ is the unit ball.  
\end{theorem}

\begin{corollary}\label{c:sharpsobolev}
For any $n\geq 2$, $p\geq 1$, $\alpha \geq 0$ and $\beta > \frac{n}{p}-\frac{n-1}{\alpha}$ there is $v\in C^{0,\alpha}\left( B_1, \R^{n} \right)$ with $\mbox{{\rm deg}}(v,B_1,\cdot) \notin W^{\beta,p}$. 
\end{corollary}

The case of the endpoints is certainly more subtle. Indeed, if $v\in C^{0,1}$ and $\Omega$ is a bounded Lipschitz domain, then the area formula and elementary considerations in degree theory imply that $\text{deg}\, v\in BV$ (the space of functions of bounded variation). In fact, with a little help from the theory of $BV$ functions and Caccioppoli sets, the latter statement can be shown even under the more technical assumption
that the  $n-1$-dimensional Hausdorff measure of $\partial \Omega$ is finite. Therefore:
\begin{itemize}
\item $\text{deg} (v, \Omega, \cdot) \in L^{n/(n-1)}$, by the Sobolev embedding of $BV (\mathbb R^n)$, hence showing that the endpoint $(\beta, p) = (0, \frac{n}{n-1})$ could be included if we assume that $\partial \Omega$ has finite $n-1$-dimensional measure;
\item since the degree takes integer values and vanishes on $\mathbb R^n \setminus v (\overline{\Omega})$, it belongs to $W^{1,1}$ only if it vanishes identically: hence, even assuming that $\partial \Omega$ has finite $n-1$-dimensional measure, the endpoint $(\beta, p) = (1,1)$ can be included only if we replace $W^{1,1}$ with $BV$.  
\end{itemize}

\subsection{Acknowledgments} The research of both authors has been supported by the grant $200021\_159403$ of the Swiss National Foundation.

\section{First estimate and change of variables}\label{s:elementary}

The starting point of Olbermann's proof is the classical change of variable formula 
\begin{equation}\label{e:Olb}
\int_{\R^{n}}\varphi(y)\, \text{deg}(v,\Omega,y)\,dy = \int_{\Omega} \varphi(v(x))\, \text{det}Dv(x)\,dx\,,
\end{equation}
which is valid if $v$ is regular enough (compare e.g. \cite{Fed}). By representing the integrand $\varphi(v(x))\text{det}Dv(x)$ as a sum of weakly defined Jacobian determinants, using Stokes theorem and tools from interpolation theory Olbermann manages to bound the right hand side of \eqref{e:Olb} by a (suitable power of the) $C^{0,\alpha}$ norm of $v$ and the $L^{p'}$ norm of $\varphi$, where $\alpha$ is as above and $p'$ is conjugate to $p$. In fact, implicit in his proof is the estimate \eqref{e:Olbermannestimate} below, which will play a crucial role for us as well. On the other hand our elementary argument yields immediately, as a byproduct, that the degree is an $L^1$ function and thus we do not have to resort to any weak notion of Jacobian determinant. In passing we also get a simple proof of Z\"ust's result, together with an appropriate generalization. 

\begin{theorem}\label{l:Olbermann}
Let $\Omega\subset \R^{n}$, $n$ and $d$ be as in Theorem \ref{t:main}. Assume that $v= (v^{1}, \ldots, v^{n})$ is a continuous map $v: \Omega \to \mathbb R^n$ for which $v^i \in C^{0, \alpha_i}$. If $\sum_i \alpha_i > d$, then $\mbox{{\rm deg}}\, (v, \Omega, \cdot)\in L^1$ and 
\begin{equation}\label{e:Zuest}
\|\mbox{{\rm deg}}\, (v, \Omega, \cdot)\|_{L^1} \leq C (\Omega, n, \alpha_1, \ldots, \alpha_n) \prod_{i=1}^{n} [v^i]_{C^{0, \alpha_i}}\, .
\end{equation}
If in addition $\alpha = \min_i \alpha_i > \frac{d}{n}$, then for any $\psi \in C^{1}\left (\R^{n},\R^{n}\right )$ we have
\begin{equation}\label{e:Olbermannestimate}
\left| \int_{\R^{n}}\mbox{{\rm deg}}\, \left (v,\Omega,y\right )\mbox{{\rm div}}\, \psi(y)\,dy\right|\leq C (\Omega, n, \alpha, \gamma)  [v]_{C^{0,\alpha}(\Omega)}^{n-1+\gamma}[\psi]_{C^{0,\gamma}(B_R)}\,,
\end{equation}
where $\gamma\in (0,1)$ is such that $(n-1+\gamma)\alpha > d$ and $R>0$ such that $\overline{v\left (\Omega\right )} \subset B_R(0)$.
\end{theorem}

\subsection{Two technical lemmas} We record here two simple facts related to the dimension of $\partial \Omega$. 

\begin{lemma}\label{l:integrabilityofdist}
 Let $\Omega\subset \R^{n}$ be a bounded open set with $d:=\mbox{{\rm dim}}_{b}\, (\partial\Omega) <n$. Then for any $\varepsilon> 0$  the function $\mbox{{\rm dist}}\, (x,\partial\Omega)^{d+\varepsilon-n}$ is integrable.
\end{lemma}
\begin{proof}
 Fix $0<\varepsilon<n-d$ and let $W$ be the Whitney decomposition of $\Omega$ and let $W_k:= \{ Q\in W: Q \text{ cube of sidelength } 2^{-k}\}$. Then 
 \begin{enumerate}
  \item dist($Q, \partial \Omega) \geq 2^{-k}\sqrt{n}$ for any $Q\in W_k$ and 
  \item there exists $C\equiv C(\varepsilon)>0 $ such that $\# W_k \leq C2^{k(d+\varepsilon/2)}$ for any $k\in \N$ (cf. Theorem 3.12 in \cite{MartioVuorinen}).
 \end{enumerate}
 
 Since $\mathring{Q} \cap  \mathring{Q'} = \emptyset$ for any $Q\neq Q'$ we have 
 \begin{align*}
 \int_{\Omega}\text{dist}(x,\partial \Omega)^{d+\varepsilon-n}\,dx &= \sum_{k\geq1}\sum_{Q\in W_k} \int_{Q}\text{dist}(x,\partial\Omega)^{d+\varepsilon-n}\,dx\leq C(n)\sum_{k\geq1}\sum_{Q\in W_k} \mathcal L^{n}\left( Q \right)2^{-k(d+\varepsilon-n)} \\
  &\leq C(n,\varepsilon)\sum_{k\geq1}2^{k(d+\varepsilon/2)}2^{-k(d+\varepsilon)} \leq C(n,\varepsilon)<+\infty\,.\qedhere
 \end{align*}
\end{proof}

\begin{lemma}\label{l:measure_zero}
If $v$ and $\Omega$ are as in Theorem \ref{l:Olbermann} then $v (\partial \Omega)$ is a Lebesgue-null set. 
\end{lemma}
\begin{proof}
Fix a positive $\delta \leq \sum_{i=1}^{n}\alpha_i -d$. For any $\varepsilon >0$ there is a covering of $\partial \Omega$ with balls $B_{r_i} (x_i)$ such that $\sum_i r_i^{d+\delta} \leq  (\mathcal{H}^{d+\delta} (\partial \Omega)) + \varepsilon = \varepsilon$ and $r_i \leq 1$, where $\mathcal{H}^\omega$ denotes the $\omega$-dimensional Hausdorff measure. 
Observe that $v (B_{r_i} (x_i))$ is contained in a box $Q_i =I^i_1\times \ldots \times I^i_n$, where each interval $I^i_j$ has length at most $(2 r_i)^{\alpha_j} [v^{j}]_{C^{0, \alpha_j}}$. Thus 
\[
|v (\partial \Omega)| \leq \sum_i |Q_i|
\leq C \prod_{j=1}^{n}[v^{j}]_{C^{0,\alpha_j}}\sum_i r_i^{\alpha_1 + \ldots + \alpha_n} \leq C(v) \varepsilon \sup_i r_i^{\alpha_1+\ldots + \alpha_n -d - \delta} \leq C \varepsilon\, .
\] 
Letting $\varepsilon\to 0$ we conclude the proof. 
\end{proof}

\subsection{Proof of Theorem \ref{l:Olbermann}} First of all recall that the degree depends only upon the values of $v$ at the boundary. We wish therefore to find a suitable extension $\tilde{v}$ of $v$ which is smooth in the interior and satisfies suitable estimates on the derivatives. For $k=0,1,\ldots$ set 
\begin{equation}\label{e:Ak}
A_k := \{ x\in  \Omega: \text{dist}(x,\partial \Omega) >2^{-k}\}\,
\end{equation}
and define $D_0 := A_1$, $D_k := A_{k+1}\setminus \bar A_{k-1}$ for $k=1,2,\ldots$. Fix a partition of unity $\{ \chi_k\}_{k\geq1}$ subordinate to the cover $\{D_k\}_{k\geq 0}$, i.e. 
\[ 0\leq \chi_k\leq 1,\quad \text{supp}\chi_k \subset D_k,\quad \sum_{k=0}^{+\infty} \chi_k = 1 \, \text{ on } \Omega\,.\]
Observe that each point $x\in \Omega$ has an open neighbourhood $U\subset\Omega$ on which at most three $\chi_k$ are non zero. Next fix a standard symmetric mollifier $\varphi$ with support contained in the ball of radius $1$ and define the functions $v_k:D_k\to \R^{n}$ by the convolution $\displaystyle v_k(x) := \varphi_{2^{-(k+1)}}\ast v (x)$. Finally, set 
\[ \tilde v := \sum_{k=0}^{+\infty} \chi_kv_k\,.\]
We have $\tilde v\in C^{\infty}\left( \Omega,\R^{n} \right)$ and we claim that for every $x\in \Omega$ 
\begin{align}
 |\nabla \tilde v^{i}(x) | &\leq C\text{dist}(x,\partial \Omega)^{\alpha_i-1}[v^{i}]_{C^{0,\alpha_i}\left( \Omega \right)}\quad \text{ for all } i \label{e:ineq1}\,.
\end{align}
By standard estimates 
\[ |\nabla v^{i}_k(y) | \leq C \left(2^{-(k+1)}\right)^{\alpha_i-1}[v^{i}]_{C^{0,\alpha_i}}\,,\quad \text{ whenever } y\in D_k\, .\]
Moreover, since $\sum \nabla \chi_k = 0$ and $| \nabla \chi_k |\leq C2^{k}$ we get 
\[ |\nabla \tilde v^{i}(x) | \leq \sum_{k=k_1}^{k_3}|\nabla \chi _k ||v_k^{i}(x) - v^{i}(x)| +C \left( 2^{-(k_3+1)} \right)^{\alpha_i-1}[v^{i}]_{C^{0,\alpha_i}}\leq C\left(2^{-(k_3+1)}\right)^{\alpha_i-1}[v^{i}]_{C^{0,\alpha_i}}\,.\]
Next, notice that $|\text{deg}\, (v, \Omega, y)| = |\text{deg} (\tilde{v}, \Omega, y)|$  is bounded by the number of preimages $N (y)$ in $\Omega$ through $\tilde{v}$ whenever $y\notin v (\partial \Omega)$. Since $v (\partial \Omega)$ is a null set, by the area formula, \eqref{e:ineq1} and Lemma \ref{l:integrabilityofdist} we have
\[
\int_{\R^n} N (y) dy = \int_\Omega |{\rm det} D\tilde v (x)|\, dx \leq C \prod_{i=1}^{n} [v^i]_{C^{0, \alpha_i}} \int_\Omega \text{dist}(x,\partial \Omega)^{\sum_i \alpha_i-n}\,dx
\leq C \prod_{i=1}^{n} [v^i]_{C^{0, \alpha_i}}\, .
\]
Next, fix a $C^1$ test field $\psi$ as in the second part of the statement and let $\alpha = \min_i \alpha_i$. Define the maps $\tilde V_{j} = (\tilde v^1,\ldots,\tilde v^{j-1}, \psi^j\circ \tilde v, \tilde v^{j+1},\ldots,\tilde v^n)$ 
and the corresponding $V_j = (v^1, \ldots, v^{j-1}, \psi^j \circ v, v^{j+1}, \ldots, v^n)$ for $j=1,\ldots, n$. 
In particular it follows $\sum_{j=1}^{n}\text{det}\,  D \tilde V_{j} = (\text{div}\, \psi) \circ \tilde v\, \text{det}\, D\tilde v$. 

Let $\Omega_k$ be smooth domains compactly contained in $\Omega$ so that\footnote{Let $A_k$ be the sets in \eqref{e:Ak} and ${\mathbf 1}_{A_k}$ their indicator functions, consider the mollifications $\eta_k := \mathbf{1}_{A_k} \ast \varphi_{2^{-k-1}}$
and set $\Omega_k = \{\eta_k > t_k\}$ for a suitably chosen $0< t_k<1$. The regularity of $\partial \Omega_k$ follows from Sard's Lemma.} $\Omega_k \uparrow \Omega$.
By the smoothness of $\tilde{v}$ and $\psi$, we can apply the area formula and conclude
\begin{align*}
\int_{\R^{n}}\text{deg}\, (\tilde v,\Omega_k,y)\, \text{div}\, \psi(y)\,dy & = \int_{\Omega_k}\text{div}\, \psi( \tilde v(x))\, \text{det}\, D\tilde v (x)\,dx = \sum_{j=1}^{n}\int_{\Omega_k}\, \text{det}\, D\tilde V_{j}(x)\,dx\\
& = \sum_{j=1}^n \int_{\R^n} \text{deg}\, (\tilde{V}_j, \Omega_k, y)\, dy\, .
\end{align*}
Next, observe that the number $N (y)$ bounds $|\text{deg}\, (\tilde v,\Omega_k,y)|$ for every $y$ and $k$ and thus, by the dominated convergence theorem,
\[
\lim_{k\to\infty} \int_{\R^{n}}\text{deg}\, (\tilde v,\Omega_k,y)\text{div}\, \psi(y)\,dy = \int_{\R^n} \text{deg}\, (\tilde v,\Omega,y)\text{div}\, \psi(y)\,dy\, .
\]
The same argument can be applied to $\tilde{V}_j$, since $|\det D\tilde{V}_j|\leq |D\psi| |D \tilde{v}|^n$ also belongs to $L^1 (\Omega)$. Hence, passing into the limit in $k$ and using the fact that $\tilde{v}$ agrees with $v$ on $\partial \Omega$ we can conclude
\[
\int_{\R^n} \text{deg}\, (v,\Omega,y)\text{div}\, \psi(y)\,dy = \sum_j \int \text{deg}\, (V_j,\Omega,y)\, dy\, .
\]
On the other hand for each $V_j$ we have $[V_j^i]_{C^{0, \alpha}} \leq [v]_{C^{0, \alpha}}$ when $i\neq j$ and 
$[V_j^j]_{C^{0, \alpha \gamma}} \leq [\psi]_{C^{0, \gamma}} [v]^\gamma_{C^{0, \alpha}}$.
Since by our choice of $\gamma$ we have $(n-1+\gamma) \alpha > d$, we can apply \eqref{e:Zuest} to conclude
\[
\|\text{deg}\, (V_j,\Omega,\cdot )\|_{L^1} \leq C (n, \Omega, \alpha, \gamma, d) [v]_{C^{0, \alpha}}^{n-1+\gamma} [\psi]_{C^{0, \gamma}}\, .
\]

\section{Proofs of Theorem \ref{t:main} and of Corollary \ref{c:convergence}}\label{s:main}

\subsection{Direct proof of Theorem \ref{t:main} for $\beta =0$}\label{s:easier} This section follows essentially Olbermann's argument and is only added for the reader's convenience in order to show that the harmonic analysis of the next section is only needed for $\beta>0$. The key is the following proposition.

\begin{proposition}\label{p:dual2}
Let  $\Omega\subset \R^{n}$, $n$, $d$, $\alpha$ and $v$ be as in Theorem \ref{t:main} with $\|v\|_{C^0}\leq 1$ and fix $1< p<\frac{n\alpha}{d}$. Then, if we denote by $p'$ the dual exponent of $p$, we have the estimate
\begin{equation}\label{e:dual2}
\left|\int \mbox{\rm deg}\,  (v, \Omega, y)\, \varphi (y)\, dy \right| \leq C (\Omega, n,d,\alpha, p. \beta)[ v]_{C^{0,\alpha}}^{\frac{n}{p}} \|\varphi\|_{L^{p'}} \qquad
\forall \varphi \in C^\infty_c (\mathbb R^n)\, .
\end{equation}
\end{proposition}

The case $\beta=0$ of Theorem \ref{t:main} then follows easily when $\|v\|_0 \leq 1$: just take the supremum over $\varphi\in C^\infty_c \cap\{ \|\varphi\|_{L^{p'}} \leq 1\}$ in \eqref{e:dual2} and
use the density of $C^\infty_c$ in $L^{p'}$ together with the usual duality $(L^p)^* = L^{p'}$.
To remove the assumption that $\|v\|_0 \leq 1$ it suffices, for a general nonzero $v$, to consider the normalization $v/\|v\|_0$ and compare its degree to that of $v$ with an obvious scaling argument (cf. Section \ref{s:Bessel} below where this argument is repeated with more details). The extension to $p=1$ follows because $\text{deg} (v,\Omega,  \cdot)$ is supported in the bounded set $v (\Omega)$, whose diameter can be estimated using the H\"older norm of the function $v$. 
We are thus left to show \eqref{e:dual2}. Fix $\varphi$ and consider the potential theoretic solution $\zeta$ of 
\[
- \Delta \zeta = \varphi\, .
\]
By classical Calderon-Zygmund estimates we have $\|\zeta\|_{W^{2,p'} (B_2)} \leq C \|\varphi\|_{L^{p'}}$. So, if we
set $\psi = - \nabla \zeta$, we conclude ${\rm div}\, \psi = \varphi$ on $B_2$ and, from the Sobolev embedding, 
$[\psi]_{C^{0, \gamma} (B_2)} \leq C\|v\|_{L^{p'}}$, where $\gamma = 1- \frac{n}{p'} = 1 - n + \frac{n}{p} > 1 - n + \frac{d}{\alpha}$. Since $\text{deg} (v, \Omega, \cdot)$ is supported in $B_2$, we can apply Theorem \ref{l:Olbermann} to conclude \eqref{e:dual2}. 

\subsection{Bessel potential spaces when $\beta>0$}\label{s:Bessel}
Rather than showing estimate \eqref{e:degree-estimate-1} we will show, for the exponents in the ranges $1< p<\frac{n\alpha}{d}$ and $0\leq \beta <\frac{n}{p}-\frac{d}{\alpha}$, the slightly different estimate
 \begin{equation}\label{e:degree-estimate}
  \| \text{deg}(v,\Omega,\cdot)\|_{\mathcal H^{\beta,p}} \leq C \| v\|_{C^{0,\alpha}}^{\frac{n}{p}-\beta} \qquad \mbox{when $\|v\|_{C^{0}} \leq  1$,}
 \end{equation}
where $\mathcal H^{\beta,p}\left( \R^{n} \right)$ is the Bessel potential space (see below for the relevant definition). Recall (see e.g. the classical textbook of Triebel \cite{Triebel}) that the spaces $W^{\beta, p}$ and $\mathcal{H}^{\beta, p}$ correspond, respectively, to the Triebel-Lizorkin spaces
$F^{p,p}_\beta$ and $F^{p,2}_\beta$. Since we have the continuous embedding $F^{p,q}_\beta \subset F^{p, q'}_{\beta- \varepsilon}$ for every $q,q'$ and every $\varepsilon >0$, we get as a corollary of \eqref{e:degree-estimate} the estimate 
\begin{equation}\label{e:degree-estimate-2}
 \|\text{deg}(v,\Omega,\cdot)\|_{W^{\beta,p}} \leq C\|v\|_{C^{0,\alpha}}^{\frac{n}{p}-\beta} \qquad \mbox{when $\|v\|_0 \leq 1$.}
\end{equation}

From \eqref{e:degree-estimate-2} it follows by scaling that for any nonzero $v$ as in Theorem \ref{t:main} we have
\begin{equation}
 [\text{deg}(v,\Omega,\cdot)]_{W^{\beta,p}} = \|v\|_{C^{0}}^{\frac{n}{p}-\beta }\left[\text{deg}\left(\frac{v}{\|v\|_{C^{0}}},\Omega,\cdot\right)\right]_{W^{\beta,p}} \leq C \|v\|_{C^{0,\alpha}}^{\frac{n}{p}-\beta}\,. 
\end{equation}
Apply the latter estimate to $\tilde v:= v-v(x_0)$ for some $x_0\in \Omega$. Since $\text{deg}\, (\tilde v, \Omega, y) = \text{deg}\, (v, \Omega, y +v (x_0))$ and $\|\tilde v\|_0 \leq C (\Omega, \alpha) [v]_{C^{0, \alpha}}$ we recover \eqref{e:degree-estimate-1}.

Recall that the Bessel potential of degree $\beta > 0$ is the $L^1$ function $J_\beta$ such that $\hat J_\beta(\xi) = \left( 1+4\pi^{2}|\xi|^{2} \right)^{-\beta/2}$ (where $\hat{h}$ denotes the Fourier transform of $h$). The convolution with $J_\beta$ defines a continuous linear map
$\mathcal{J}_\beta: L^p \to L^p$ and can be regarded as the pseudodifferential operator $(\text{Id} -\Delta)^{-\beta/2}$. In particular
\begin{equation}\label{e:inverse}
(\text{Id} -\Delta)\mathcal J_2\varphi = \varphi\qquad \forall \varphi\in C^\infty_c (\R^{n})\, .
\end{equation}
Concerning the Bessel potential space $\mathcal{H}^{\beta, p}$ we will need the following facts (cf. again \cite{Triebel}):
\begin{itemize}
\item[(F1)] $f\in \mathcal{H}^{\beta, p}$ if and only if there is $g\in L^p$ with $f = \mathcal{J}_\beta (g)$; such $g$ is unique and $\|f\|_{\mathcal H^{\beta,p}} = \|g\|_{L^{p}}$;
\item[(F2)] $(\mathcal{H}^{\beta, p}, \|\cdot\|_{\mathcal{H}^{\beta, p}})$ is a separable reflexive Banach space for any $p\in ]1, \infty[$ and $C^\infty_c (\mathbb R^n)$ is dense in it; 
\item[(F3)]  if $\beta p> n$ and $p\geq 2$ we have the continuous inclusion $\mathcal{H}^{\beta, p} \subset W^{\beta, p}$ and hence, by Morrey's embedding, $\mathcal H^{\beta,p} \subset C^{0,\gamma}$ with $\gamma = (\beta p - n)/p$. 
\end{itemize}
The idea of the proof of Theorem \ref{t:main} is to show that $\text{deg}(v,\Omega,\cdot)$ is an element of the dual of $(\mathcal{H}^{\beta, p})^*$ and to use the reflexivity property in (F2). As usual, $(\mathcal{H}^{\beta, p})^*$ denotes the Banach space of bounded linear functionals $L: \mathcal{H}^{\beta, p} \to \mathbb R$ endowed with the dual norm $\|\cdot\|_{(\mathcal{H}^{\beta,p})^*}$.
Moreover, since $C^\infty_c (\mathbb R^n)$ is dense in $\mathcal{H}^{\beta, p}$, we clearly have
\begin{equation}\label{e:dualnorm}
 \|L\|_{(\mathcal{H}^{\beta,p})^*} := \sup \left\{L (u) :  u\in C^\infty_c (\R^n) \;\mbox{and}\; \|u\|_{\mathcal H^{\beta,p}} \leq 1\right\} \,.
\end{equation} 
Of course $\left(\mathcal{H}^{\beta, p}\right)^{*}$ is a subspace of the space of tempered distributions and we can consider $C^\infty_c$ as a subset
of $\left(\mathcal{H}^{\beta, p}\right)^{*}$ via the identification of any element $\varphi\in C^\infty_c$ with the linear functional $u\mapsto \int \varphi u$. We then have the following standard consequence of distribution theory

\begin{lemma}\label{l:density}
$C^\infty_c$ is strongly dense in $(\mathcal{H}^{\beta, p})^*$ if $p\in ]1, \infty[$.
\end{lemma} 
\begin{proof}
Let $H$ be the closure of $C^\infty_c$ in the norm $\|\cdot\|_{\left( \mathcal H^{\beta,p} \right)^{*}}$. If $H$ were a strict subset of $\mathcal{H}^{\beta, p}$, then 
by Hahn-Banach there would be a nontrivial linear functional $L': (\mathcal{H}^{\beta, p})^* \to \mathbb R$ wich vanishes on $H$. 
By reflexivity $L'$ is given by an element $u\in \mathcal{H}^{\beta, p}$, which must therefore be nonzero. Since however $L'$ vanishes on 
$H$, we conclude
\[
\int u \varphi = 0 \qquad \forall \varphi \in C^\infty_c\, .
\]
Since $u\in L^p$, the latter implies that $u\equiv 0$, which is a contradiction. 
\end{proof}

\eqref{e:degree-estimate} is then a consequence of the following natural generalization of  Proposition \ref{p:dual2}.

\begin{proposition}\label{p:dual}
Let  $\Omega\subset \R^{n}$, $n$, $d$, $\alpha$ and $v$ be as in Theorem \ref{t:main} with the additional assumption $\|v\|_0 \leq 1$ and fix $1< p<\frac{n\alpha}{d}$ and $0< \beta <\frac{n}{p}-\frac{d}{\alpha}$. Then we have the estimate
\begin{equation}\label{e:dual}
\left|\int_{\R^{n}} \mbox{{\rm deg}}\, (v, \Omega, y)\, \varphi (y)\, dy \right| \leq C (\Omega, n,d,\alpha, p. \beta)[ v]_{C^{0,\alpha}}^{\frac{n}{p}-\beta} \|\varphi\|_{(\mathcal{H}^{\beta, p})^*} \qquad
\forall \varphi \in C^\infty_c (\mathbb R^n)\, .
\end{equation}
\end{proposition}

We will prove Proposition \ref{p:dual} in the next section. Assuming it, we now show \eqref{e:degree-estimate}. Consider
the linear functional $L'': C^\infty_c \to \mathbb R$ given by
\[
L'' (\varphi) := \int_{\R^{n}} \text{deg} (v, \Omega, y)\, \varphi (y)\, dy \, .
\]
By Lemma \ref{l:density} and \eqref{e:dual}, $L'$ extends to a unique bounded linear functional $\mathcal{L} : (\mathcal{H}^{\beta, p})^* \to \mathbb R$ and moreover
\[
\|\mathcal{L}\|_{(\mathcal{H}^{\beta,p})^{**}} \leq C \| v\|_{C^{0,\alpha}}^{\frac{n}{p}-\beta}\, .
\]
By reflexivity $\mathcal{L}$ is represented by an element $u\in \mathcal{H}^{\beta, p}$ such that
$\|u\|_{\mathcal{H}^{\beta, p}} = \|\mathcal{L}\|_{(\mathcal{H}^{\beta,p})^{**}}$. This means
\[
\int_{\R^{n}} u (y)\varphi (y)\, dy= L'' (\varphi) = \int_{\R^{n}} \text{deg} (v, \Omega, y)\, \varphi (y)\, dy
\]
for every $\varphi\in C^\infty_c$. Since however both $\text{deg} (v, \Omega, \cdot)$ and $u$ are $L^p$ functions, they must
coincide. Hence 
\[
\|\text{deg} (v, \Omega, \cdot)\|_{\mathcal{H}^{\beta, p}}= \|u\|_{\mathcal{H}^{\beta, p}} = \|\mathcal{L}\|_{(\mathcal{H}^{\beta,p})^{**}} \leq C \| v\|_{C^{0,\alpha}}^{\frac{n}{p}-\beta}\, .
\]

\subsection{Proof of Proposition \ref{p:dual}} In order to prove the estimate \eqref{e:dual}, we will invoke 
property \eqref{e:Olbermannestimate} after representing $\varphi$ as the divergence of a suitable vector field,
which is the purpose of the following lemma.

\begin{lemma}\label{l:goodmap}
Let $\varphi \in C^{\infty}_c\left (\R^{n}\right )$ and assume $1<p<\frac{n}{n-1}$ and $\beta \in ]0,1[$ with $(1-\beta)p'>n$ (where $p'$ is the dual exponent of $p$). Then there exists $\psi\in C^{\infty}\left (\R^{n},\R^{n}\right )$ such that 
\[\mbox{{\rm div}}\,  \psi = \varphi \quad \text { on } B_2\]
and, setting $\gamma = 1-\beta-n/p'$,
\[ \|\psi\|_{C^{0,\gamma}(B_2)} \leq C (n, \gamma, \beta, p) \| \varphi\|_{(\mathcal{H}^{\beta,p})^*}\, .\]
\end{lemma}
\begin{proof}
First of all observe that the condition $1<p<\frac{n}{n-1}$ implies $p'> n$ so that the condition on $\beta$ makes sense. Set $\zeta = \mathcal J_2 \varphi $. Then $\zeta \in C^{\infty}\left (\R^{n}\right )$ satisfies 
\begin{equation}
\label{e:lemmagoodmap1}
-\Delta \zeta +\zeta = \varphi \quad \text { on } \R^{n}
\end{equation}
and we claim that
\begin{equation}
\label{e:lemmagoodmap2}
\|\zeta \|_{C^{1,\gamma}\left (\R^{n}\right )} \leq C\|\varphi\|_{(\mathcal{H}^{\beta,p})^*}\,.
\end{equation}
Indeed, set $f=\mathcal J_\beta \varphi \in L^{p'}\left (\R^{n}\right )$ with $\|f\|_{p'}\leq C\|\varphi\|_{p'}<+\infty$, and $\mathcal J_{2-\beta} f = \mathcal J_2 \varphi = \zeta$. Observe that for any $g\in C^{\infty}_c\left (\R^{n}\right )$ with $\|g\|_p\leq 1$ we have 
\[ 
\int_{\R^{n}} f g\,dx = \int_{\R^{n}} \varphi \mathcal J_\beta g \,dx \leq \|\varphi\|_{(\mathcal{H}^{\beta,p})^*} \|\mathcal J_\beta g\|_{\mathcal H^{\beta,p}} \leq \|\varphi\|_{(\mathcal{H}^{\beta,p})^*}\,. 
\]
Taking the supremum over such functions $g$ yields $\|\zeta\|_{\mathcal H^{2-\beta,p'}}\leq \|\varphi\|_{(\mathcal{H}^{\beta,p})^*}$. Claim \eqref{e:lemmagoodmap2} then follows by the continuous embedding (F3).

Now fix a cutoff function $\eta \in C^{\infty}_c\left (\R^{n}\right ) $ with $\eta \equiv 1 $ on $B_2$ and ${\rm spt}\, \eta \subset B_{3}$ and denote by $\bar \zeta$ the classical potential theoretic solution of $- \Delta \bar \zeta = \zeta \eta$. By classical estimates (cf. \cite[Chapter 4]{GT}) we get 
\begin{equation}\label{e:lemmagoodmap3}\|\nabla \bar{\zeta}\|_{\gamma,B_2} \leq C \|\zeta \eta\|_{\gamma,B_{4}} \leq C \|\zeta\|_{1+\gamma,\R^{n}} \leq C \|\varphi\|_{(\mathcal{H}^{\beta,p})^*}\,.\end{equation}
Finally we set $\psi := - \nabla( \bar\zeta + \zeta)$. 
Then by \eqref{e:lemmagoodmap1} 
\[\text {div}\, \psi = \zeta -\Delta \zeta = \varphi \quad \text { on } B_2\,,\]
and by \eqref{e:lemmagoodmap2} and \eqref{e:lemmagoodmap3}
\[ \|\psi \|_{C^{0,\gamma}\left (B_2\right )} \leq C \|\varphi\|_{(\mathcal{H}^{\beta,p})^*}\,.\qedhere\]

\end{proof}
The proof of \eqref{e:dual} is now an immediate corollary of Theorem \ref{l:Olbermann} and Lemma \ref{l:goodmap}.

\subsection{Proof of Corollary \ref{c:convergence}} Note that:
\begin{itemize}
\item $\mbox{deg}\, (v_k, \Omega, \cdot)$ converges pointwise to $\mbox{deg}\, (v, \Omega, \cdot)$ on $\mathbb R^n \setminus v(\partial \Omega)$;
\item $v (\partial \Omega)$ is a Lebesgue null set;
\item For any pair $(\beta', p)$ as in \eqref{e:range} with $\beta'> \beta$ we have a uniform bound on $\|\mbox{deg}\, (v_k, \Omega, \cdot)\|_{W^{\beta', p}}$;
\item There is $R>0$ such that $\|v\|_0, \sup_k \|v_k\|_0 <R$ and thus the functions $\mbox{deg}\, (v_k, \Omega, \cdot)$ and $\mbox{deg}\, (v, \Omega, \cdot)$ all vanish  outside $B_R (0)$.
\end{itemize}
Thus the strong convergence claimed in Corollary \ref{c:convergence} follows from the compact embedding of $W^{\beta', p} (B_R (0))$ into $W^{\beta, p} (B_R (0))$.

\section{Proof of Theorem \ref{t:sharpmoredim}}
To prove Theorem \ref{t:sharpmoredim} we construct, for $p\in[1,\frac{n}{n-1}[$ and $\alpha < \frac{p(n-1)}{n}$, a map $v \in C^{0,\alpha}\left( B_1,\R^{n} \right)$ with deg$(v,B_1,\cdot)\notin L^{p}\left( \R^{n} \right)$ by explicitly defining it on the boundary $\partial B_1$.  Since the support of degree is bounded, clearly our map cannot belong to $L^{p^*}$ for any $p^*$ larger than such $p$. Any $C^{0,\alpha}$ extension of $v$ to the whole $\overline B_1$ then does the job, since the degree only depends on the values on the boundary of the domain.
The image  $v(\partial B_1)$ will be the union of countably many spheres $S_k$ with decreasing radii $r_k$. Each sphere $S_k$ will be circled a certain $c_k$ times in each direction. The goal is to choose the radii $r_k$ and the number of circlings $c_k$ in such a way that  $v$ is H\"older continuous with exponent $\alpha<\frac{p(n-1)}{n}$, but deg$(v,B_1,\cdot) \notin L^{p}\left( \R^{n} \right)$.\\
Given $p\in [1,\frac{n}{n-1}[$ we define a partition $\{I_k\}_{k\geq1}$ of the interval $[-\pi,\pi[$ as follows. \\For $k\geq 1$ define the numbers 

 \begin{equation}\label{e:setlength}
   |I_k| = c(n,p)k^{-\left( \frac{n-1}{n}+\frac{1}{p(n-1)} \right)}\,,
 \end{equation}
where the constant $c(n,p)$ is determined by the condition $\displaystyle \sum_{k\geq 1} |I_k| = 2\pi $. The sets $I_k$ are then defined by 
 \begin{align}\label{e:sets} 
 &I_1 = \left[\frac{-|I_1|}{2},\frac{|I_1|}{2}\right[\,,\quad \text{ and }\\
  I_k = \left[\frac{-\sum_{i=1}^{k}|I_i|}{2},\frac{-\sum_{i=1}^{k-1}|I_i|}{2}\right[ &\cup \left[\frac{\sum_{i=1}^{k-1}|I_i|}{2},\frac{\sum_{i=1}^{k}|I_i|}{2}\right[\,, \quad \text{ for } k\geq 2\,.
 \end{align}
Note that in this way the length of the set $I_k$ coincides with the number $|I_k|$. 

For brevity (and clarity) we introduce the following map $\Phi:[-\pi,\pi[\times[0,\pi]^{n-2} \to \R^{n}$ which is the usual (almost) parametrization of the sphere:
  \[ \Phi(\theta_1,\ldots,\theta_{n-1}) = (\cos\theta_1,\sin\theta_1\cos\theta_2,\ldots ,\sin\theta_1\cdot\ldots\cdot\sin\theta_{n-2}\cos\theta_{n-1},\sin\theta_1\cdot\ldots\cdot\sin\theta_{n-2}\sin\theta_{n-1})\,\]
The sets $I_k$ naturally give a decomposition of the sphere $\partial B_1$ into
 \[ J_k := \Phi (I_k \times [0,\pi]^{n-2})\,.\]
In the rest of the proof by a slight abuse of notation we identify $J_k$ with $I_k \times [0, \pi]^{n-2}$ and define $v$ over the latter domains: the map $\Phi$ is a parametrization on
$[-\pi, \pi[\times ]0, \pi[^{n-2}$, however $v$ will be constant on the set $[-\pi, \pi[ \times \partial ([0, \pi]^{n-2})$ and hence it will induce a well-defined map over the sphere. 

  For a given $\alpha< \frac{p(n-1)}{n}$ we then choose a number  
 \begin{equation}\label{e:q2}
  q > \frac{\alpha p(n-1)^{2}}{n(p(n-1)-\alpha n)}
 \end{equation}
 and define the radii 
 \begin{equation}\label{e:radii2}
   r_k = k^{-q}\quad \text{ for } k\geq 1\,.
 \end{equation}
 We then set the number of circlings to be 
 \begin{equation}\label{e:circlings2}
  c_k =  k^{\frac{qn-1}{p(n-1)}}\,,
 \end{equation}
which with an appropriate choice of $q$ in \eqref{e:q2} is a natural number for all $k$. 
For notational convenience we introduce the reparametrization
  \[ \Theta(\theta) = \begin{cases}
                                \frac{2\pi}{|I_1|}\theta+\pi \quad \text{ when } \theta \in I_1\\
                                \frac{4\pi(c_k+\frac{1}{2})}{|I_k|}\theta +\phi_k (\theta) \quad \text{ when }\theta\in I_k, k\geq 2\,,
                               \end{cases}\]
  where $\phi_k$ are phases defined by
\begin{equation}\label{e:phase} \phi_k(\theta) = \pi+ \pi(2c_k+1)\left(1-\text{sgn}(\theta)\frac{\sum_{i=1}^{k}|I_i|}{|I_k|}\right)\,, \end{equation}
which will ensure the continuity of the map. 

We then introduce the centerpoints of the spheres
  \[ x_k =\begin{cases}
           (r_1,0,\ldots,0) \quad \text{ for } k = 1\,,\\
           \left( r_k+2\sum_{i=1}^{k-1}r_i,0,\ldots,0 \right) \quad \text{ for } k\geq 2\,.
          \end{cases}\]
 Finally we define 
  \begin{equation}\label{e:themapv}
  v(\theta_1,\ldots,\theta_{n-1}) = x_k+r_k\Phi\left( \Theta(\theta_1),c_k\theta_2,\ldots,c_k\theta_{n-1} \right)\quad \text{ when } \theta_1\in I_k\,.\\
   \end{equation}
 The image $v(\partial B_1)$ decomposes into the union of countably many spheres $S_k= v(J_k)$ of radius $r_k$ and centers $x_k$. The intersection of any $S_k$ with $S_{k+1}$ only contains the northpole of $S_k$ (respectively the southpole of $S_{k+1}$), see Figure 1. 
 \begin{figure}[h]
  \centering
 \includegraphics[width=0.65\textwidth]{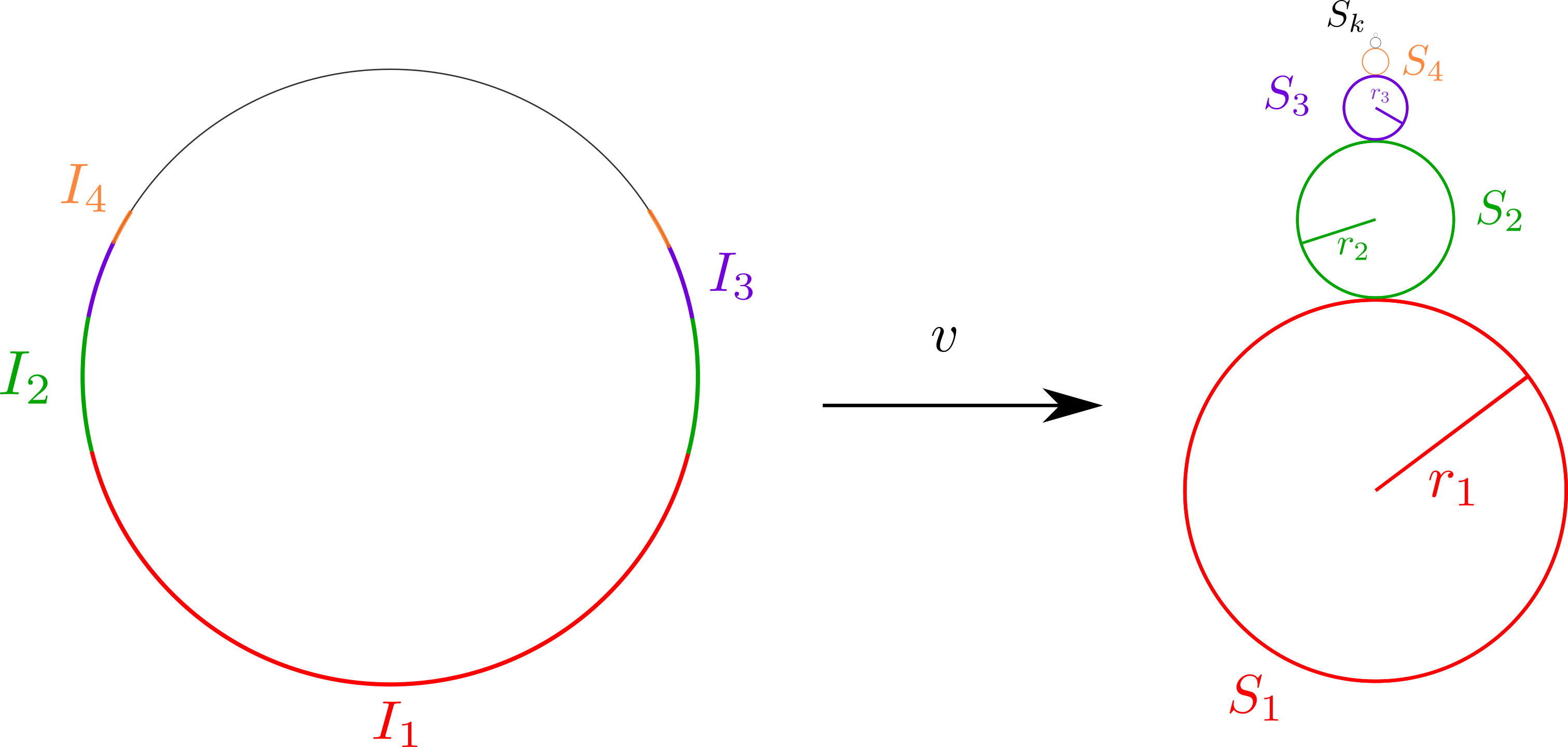}
 \captionof{figure}{The map $v$ for $n=2$: it goes around $S_1$ once and traverses every $S_k$ $2c_k+1$ times ($c_k+1/2$ times on each component of $I_k$)}
 \end{figure}

  We claim that  $v\in C^{0,\alpha}\left( \partial B_1,\R^{n} \right)$. First observe that the choice of $q$ in \eqref{e:q2} implies 
 \begin{equation}\label{e:hoeldercondition}
  r_k \leq \left( \frac{|I_k|}{c_k} \right)^{\alpha}\,.
 \end{equation}
 Indeed, this equation is equivalent to 
 \[ k^{-q}\leq k^{-\alpha\left(\frac{n-1}{n}+\frac{qn}{p(n-1)}\right)}\,,\]
 which is satisfied whenever 
 \[ q\left( 1-\frac{\alpha n}{p(n-1)} \right) > \frac{\alpha (n-1)}{n}\,,\]
 i.e.
 \[ q >\frac{\alpha p (n-1)^{2}}{n(p(n-1)-\alpha n)}\,.\]
But inequality \eqref{e:hoeldercondition} guarantees the desired H\"older regularity. To see this, we first fix the angles $\theta_2,\ldots,\theta_{n-1}$ and consider variations only in the first variable. To this end we let 
\[ u(\theta) = v(\theta,\theta_2,\ldots,\theta_{n-1}) \quad \text{ for } \theta \in [-\pi,\pi)\,,\]
fix $\theta,\tilde\theta\in [-\pi,\pi)$ and consider the following cases. 
 \begin{enumerate}
  \item $\theta, \tilde\theta \in I_k$ for some $k\geq 1$. If $|\theta-\tilde\theta| \geq \frac{|I_k|}{2(c_k+1/2)}=\frac{|I_k|}{2c_k+1}$, then
  \[ \frac{|u(\theta)-u(\tilde\theta)|}{|\theta-\tilde\theta|^{\alpha}}\leq 2r_k \left(\frac{|I_k|}{2c_k+1}\right)^{-\alpha} \leq C\,,\]
  by \eqref{e:hoeldercondition}. If however $|\theta-\tilde\theta|<\frac{|I_k|}{2c_k+1}$, then 
  \[ |u(\theta)-u(\tilde\theta)| \leq \frac{4\pi r_k(c_k+\frac{1}{2})}{|I_k|}|\theta-\tilde\theta|\leq \frac{4\pi r_k(c_k+\frac{1}{2})}{|I_k|}\left( \frac{|I_k|}{2c_k+1} \right)^{1-\alpha}|\theta-\tilde\theta|^{\alpha}\leq C|\theta-\tilde\theta|^{\alpha}\,.\]
  \item $\theta\in I_{k+1}, \tilde\theta\in I_k$ for some $k\geq1$. If $|\theta-\tilde\theta|\geq |I_k|$, then 
  \[\frac{|u(\theta)-u(\tilde\theta)|}{|\theta-\tilde\theta|^{\alpha}}\leq \frac{4r_k}{|I_k|^{\alpha}}\leq C\,.\]
  If however $|\theta-\tilde\theta|< |I_k|$ then they lie in adjacent intervals and we can compare with the endpoint $\theta_* = \frac{\text{sgn}(\theta)\sum_{i=1}^{k}|I_i|}{2}$ to get 
  \[\frac{|u(\theta)-u(\tilde\theta)| }{|\theta-\tilde\theta|^{\alpha}}\leq \frac{|u(\theta)-u(\theta_*)|}{|\theta-\theta_*|^{\alpha}}+\frac{|u(\tilde\theta)-u(\theta_*)|}{|\tilde\theta-\theta_*|^{\alpha}} \leq  C\,.\]
  \item $\theta \in  I_{k+j}, \tilde\theta\in I_k$ for some $k\geq 1$ and $j\geq 2$. Clearly $|\theta - \tilde\theta | \geq \frac{1}{2}\sum_{i=1}^{j-1}|I_{k+i}| $ so 
\begin{align*}
 \frac{|u(\theta)-u(\tilde\theta)|}{|\theta-\tilde\theta|^{\alpha}}&\leq 2^{\alpha}\frac{\sum_{i=k}^{k+j} 2r_i}{\left(\sum_{i=1}^{j-1}|I_{k+i}|\right)^{\alpha}}\leq 2^{1+\alpha}\left( \frac{r_k+r_{k+j}}{|I_{k+1}|^{\alpha}}+\sum_{i=1}^{j-1}\frac{r_i}{|I_{k+i}|^{\alpha}}\right) \\&\leq C\left( \frac{r_k}{|I_k|^{\alpha}}+\sum_{i=1}^{\infty}c_i^{-\alpha} \right)\leq C(p,\alpha)\,,
\end{align*}
 if $q$ is chosen large enough. 
 \end{enumerate}
The proof of the H\"older regularity is now complete in the case $n=2$. In the more general case some extra care is needed: a similar computation yields the H\"older regularity in the variable $\theta_i$ for every  $i=2,\ldots,n-1$ but one must take into account that the map $\Phi$ is not really a parametrization of the sphere. We leave the details to the reader.  
 
 To compute the degree we introduce the natural extension $\tilde v:[0,1]\times [-\pi,\pi[\times[0,\pi]^{n-2}\to \R^{n}$ with 
  \begin{equation}\label{e:extension}
   \tilde v(r,\theta_1,\ldots,\theta_{n-1}) = x_k+r\cdot r_k\Phi\left( \Theta(\theta_1),c_k\theta_2,\ldots,c_k\theta_{n-1} \right)\quad \text{ when } \theta_1\in I_k\,, \\
  \end{equation}
 Then $\tilde v \left( [0,1]\times J_k \right)$ is a ball $B^{k}$ with boundary $\partial B^{k} = S_k$. Fix a $y\in \text{Im} (\tilde v)\setminus \tilde v(\partial B_1)$. Then there exists a unique $k\in\N$ such that $y\in B^{k}$.  We can therefore parametrize  $y$ by 
  \[ y = x_k + r\cdot r_k\Phi(\phi_1,\ldots,\phi_{n-1}),\]	
  for some $r\in [0,1]$, $\phi_i \in [0,\pi]$ for $i=1,\ldots, n-2$ and $\phi_{n-1}\in [0,2\pi)$. 
  By definition the degree is then given by 
  \[ \text{deg}(\tilde v,B_1,y)  =\sum_{x\in \tilde v^{-1}(y)} \text{sgn det} D\tilde v(x)\,.\]
  By the chain rule and the usual expression for the spherical volume element we get for a point $x=(\tilde r,\theta_1,\ldots,\theta_{n-1})$ with $\tilde v(x) =y$
   \[ \text{det}§D\tilde v( x) = r_k^{n}(r\cdot c_k)^{n-1}\frac{4\pi (c_k+\frac{1}{2})}{|I_k|}\sin^{n-2}( \Theta_1(\theta_1))\sin^{n-3}\left( c_k\theta_2 \right)\cdot\ldots\cdot\sin\left(c_k\theta_{n-2} \right)\,,\]
  hence we have to investigate the sign of the sines. To this end we observe that $\tilde v(\tilde r,\theta_1,\ldots,\theta_{n-1}§) = y$ if and only if 
  \begin{equation}\label{e:angles}
   \begin{cases}
    \tilde r &=r \\
    \Theta(\theta_1) &= \phi_1 + 2\pi m_1 \, \text{ for } m_1\in \N \cap [\frac{1}{2}-\frac{\phi_1}{2\pi}, 2c_k+\frac{3}{2}-\frac{\phi_1}{2\pi}] \\
    c_k \theta_2 &= \phi_2 + 2\pi m_2 \, \text { for } m_2 = 1,\ldots, c_k\\
    \quad \vdots & \quad  \quad\quad  \vdots \\
    c_k\theta_{n-1} &= \phi_{n-1}+2\pi m_{n-1} \, \text{ for } m_{n-1} = 1,\ldots, c_k\,.
   \end{cases}
  \end{equation}
  Since for $i=1,\ldots, n-2$ the angles $\phi_i$ satisfy $0\leq \phi_i\leq \pi$ this implies that $\text{sgn det} Dv(x) = 1$ for any $x\in \tilde v^{-1}(y)$. Consequently, with the help of \eqref{e:angles} we conclude 
  \[ \text{deg}(\tilde v,B_1,y)  =\# v^{-1}(y) \geq  2c_k^{n-1}\,.\]
  From this in turn we deduce 
  \[ \int_{\R^{n}} |\text{deg}(\tilde v, B_1,y)|^{p}\,dy \geq C \sum_{k\geq 1} r_k^{n}c_k^{p(n-1)} = C\sum_{k\geq 1} k^{-1}=+\infty\,,\]
  by the choice of $r_k$ and $c_k$ in \eqref{e:radii2} and \eqref{e:circlings2} respectively. To conclude the proof we extend $v$ by keeping its $C^{0,\alpha }$ norm to the whole $\overline B_1$, and are left with a map $v\in C^{0,\alpha}\left(B_1,\R^{n}  \right)$ such that $\text{deg} (v,B_1,\cdot ) = \text{deg} (\tilde v,B_1,\cdot) \notin L^{p}\left( \R^{n} \right)$. 
\bibliographystyle{plain}
 \bibliography{degreebib}

\end{document}